\def\IEEEtheorems{}

\let\confidentialstring=\relax

\let\headnote=\relax

\def\mykeywords{%
\ifCLASSOPTIONonecolumn
Linear systems, time-varying systems, strong structural properties,
controllability, observability;
\fi
MSC: Primary, 93B05;
Secondary, 93B07, 93C05, 15A03, 05C50%
.}

\documentclass[letterpaper,9pt,technote]{IEEEtran}
\def\confidentialstring{%
This work has been accepted for publication in the
\emph{IEEE Trans. Automat. Control};
\href{http://dx.doi.org/10.1109/TAC.2014.2320297}{DOI: 10.1109/TAC.2014.2320297}.
Copyright may be transferred without notice, after which this
version may no longer be accessible.%
}

\def\myname{Gunther Reissig, Christoph Hartung, and Ferdinand Svaricek}

\def\mytitle{Strong Structural Controllability and Observability
\ifCLASSOPTIONonecolumn\else\\\fi
of Linear Time-Varying Systems}

\usepackage{cite}
\usepackage{pgf}
\usepackage{psfrag}
\usepackage{algorithmic}
\usepackage{graphicx}
\usepackage[cmex10]{amsmath}
\let\ORGforeignlanguage\foreignlanguage
\def\foreignlanguage#1{\lowercase{\ORGforeignlanguage{#1}}}
\ifCLASSOPTIONonecolumn\ifCLASSOPTIONdraftclsnofoot\relax\else%

\fi\fi

\usepackage{gunther2e}

\ifx
\hypersetup\undefined\def\href#1#2{\texttt{#2}}
\else
\hypersetup{
pdftitle={\mytitle},
pdfauthor={Gunther Reissig, Christoph Hartung, Ferdinand Svaricek},%
pdfsubject={\confidentialstring{} \headnote},
pdfkeywords={\mykeywords},
urlcolor={\ifCLASSOPTIONonecolumn\ifCLASSOPTIONdraftclsnofoot black\else blue\fi\else black\fi}
}
\fi

\newcounter{GraphConditionCtr}
\begin{document}\bstctlcite{IEEEtranBSTCTL_AuthorNoDash}%
\makeatletter
\renewenvironment{proof}[1][\proofname]{\par
  \pushQED{\qed}%
  \normalfont \topsep6\p@\@plus6\p@\relax
  \trivlist
  \item[\hskip\labelsep
        \itshape
    #1\@addpunct{.}]\ignorespaces
}{%
  \popQED\endtrivlist\@endpefalse
}
\markboth{\hspace*{\fill}\headnote\hspace*{\fill}}%
{\hspace*{\fill}\headnote\hspace*{\fill}}%
\makeatother

\title{%
\mytitle}

\author{\myname%
\thanks{%
\ifCLASSOPTIONdraftclsnofoot\baselineskip1.5em\fi{}%
The authors are with the
University of the
Federal
Armed Forces Munich,
Dept. Aerospace Eng.,
Chair of Control Eng., %
D-85577 Neubiberg (Munich),
Germany,
\ifCLASSOPTIONdraftclsnofoot%
\texttt{\{gunther.reissig,christoph.hartung,ferdinand.svaricek\}@unibw.de}
\else
\url{http://www.reiszig.de/gunther/};
\url{http://www.unibw.de/lrt15/};
\fi
}%
\thanks{\confidentialstring{} \headnote}
}

\maketitle

\begin{abstract}
\ifCLASSOPTIONonecolumn\relax\else\boldmath\fi%
In this note we consider continuous-time systems
$\dot x(t) = A(t) x(t) + B(t) u(t)$,
$y(t) = C(t) x(t) + D(t) u(t)$
as well as discrete-time systems
$x(t+1) = A(t) x(t) + B(t) u(t)$,
$y(t) = C(t) x(t) + D(t) u(t)$
whose coefficient matrices $A$, $B$, $C$ and $D$
are not exactly known. More precisely,
all that is known about the systems is their nonzero pattern, i.e.,
the locations of the nonzero entries in the coefficient matrices.
We characterize the patterns that guarantee controllability
and observability, respectively,
for all choices of nonzero time functions at the matrix positions
defined by the pattern,
which extends a result by \person{Mayeda} and \person{Yamada}
for time-invariant systems.
As it turns out, the conditions on the patterns
for time-invariant and for time-varying discrete-time systems
coincide, provided that the underlying time interval is
sufficiently long.
In contrast, the conditions for time-varying
continuous-time systems are more restrictive than in
the time-invariant case.
\end{abstract}


\section{Introduction}
\label{s:intro}

In this note, we present novel results on controllability of the
linear discrete-time control system
\begin{equation}
\label{e:DiscreteTimeSystem}
x(t+1) =
A(t) x(t) + B(t) u(t)
\end{equation}
and of the continuous-time control system
\begin{equation}
\label{e:ContinuousTimeSystem}
\dot x(t) =
A(t) x(t) + B(t) u(t),
\end{equation}
each possibly extended by the output equation
\begin{equation}
\label{e:OutputEquation}
y(t) = C(t) x(t) + D(t) u(t).
\end{equation}
Here, the time $t$ is integer valued and real valued, respectively,
and the coefficient matrices $A(t)$, $B(t)$, $C(t)$ and $D(t)$ may be
real or complex.
If the latter matrices are all constant, the
respective systems are \begriff{time-invariant}, and otherwise they
are \begriff{time-varying}.
We remark that the same notation is used for both discrete-time and
continuous-time systems to keep the notation concise. In particular,
the symbol $t$ denotes time in both cases, and it will
always be clear from context to which of the systems
\ref{e:DiscreteTimeSystem} and \ref{e:ContinuousTimeSystem} we refer.

When such systems arise in applications, the coefficient matrices
usually depend on physical parameters and other factors. Then the
values of the entries are not known precisely, so that
system properties can not, in general, be determined with
complete certainty either.
In contrast, the nonzero pattern of the system, i.e., the locations of
the nonzero entries in its coefficient matrices, is usually completely
defined by the modeling process. That fact can be exploited to
determine \begriff{structural properties} or
\begriff{strong structural properties},
two approaches which have found wide applications. See
\cite{MayedaYamada79,Reinschke88,Murota00,DionCommaultVanderWoude03,i03diagnosis,LiuLinChen13b,LouHong12}
and the references given there.

Here we follow the strong structural approach
initiated
in
\cite{MayedaYamada79}, which
assumes that the nonzero pattern is all that is known about the
system and
seeks to characterize the patterns that guarantee
certain system properties for all choices of nonzero values at the
matrix positions defined by the pattern.
In contrast, the structural approach would
guarantee any property
only for almost all choices
\cite{Reinschke88,Murota00,DionCommaultVanderWoude03,i03diagnosis,LiuLinChen13b}.

\ifCLASSOPTIONdraftclsnofoot\else\looseness-1\fi
\person{Mayeda} and \person{Yamada} have been the first to
characterize the patterns that guarantee controllability of
time-invariant systems \cite{MayedaYamada79}.
Equivalent characterizations have been given subsequently, see
\cite{JarczykSvaricekAlt11} and the references therein, and related
problems have been investigated in
\cite{HashimotoAmemiya11,LouHong12,ChapmanMesbahi13}.
All these works consider only
time-invariant systems, despite the fact that it is often most natural to
assume that the physical parameters entering the coefficient matrices
vary over time.

In this note, we characterize the patterns that guarantee
controllability for all choices of nonzero time functions at the
matrix positions defined by the pattern, which extends the results in
\cite{MayedaYamada79} to time-varying systems.
To this end we introduce basic notation and terminology in Section
\ref{s:prelims} and review controllability results for time-invariant
systems in Section \ref{ss:ControllabilityTimeInvariant}, which
includes an algorithm to verify the conditions in the result of
\cite{MayedaYamada79}.
Patterns guaranteeing controllability for
time-varying discrete-time and continuous-time systems are
characterized in Section \ref{ss:DiscreteTimeControllability} and
\ref{ss:ContinuousTimeControllability}, respectively.
In the former case, the conditions coincide with those for
time-invariant systems if the underlying time interval is sufficiently
long, whereas in the latter case, they turn out to be more restrictive.
In Section \ref{ss:Observability} we present our results on
observability.
The conditions in our characterizations can all be
verified using the algorithm in Section
\ref{ss:ControllabilityTimeInvariant}.
Special cases of the results in the present paper have been
announced in \cite{i12str,i13str,i13strb}.

\section{Preliminaries}
\label{s:prelims}

\subsection{Basic Notation}

$\mathbb{C}$, $\mathbb{R}$ and $\mathbb{Z}$ denote the sets of complex
numbers, real numbers and integers, respectively.
$\mathbb{R}_{+}$ and $\mathbb{Z}_{+}$ denote the subsets of
non-negative elements of $\mathbb{R}$ and $\mathbb{Z}$, respectively,
and $\mathbb{N} = \mathbb{Z}_{+} \setminus \{ 0 \}$.

For $a, b \in \mathbb{R} \cup \{ \infty \}$ satisfying
$a \leq b$, the closed, open and
half-open intervals with end points $a$ and $b$ are denoted
$\intcc{a,b}$, $\intoo{a,b}$, $\intco{a,b}$, and $\intoc{a,b}$,
respectively,
e.g.
$\intco{0,\infty} = \mathbb{R}_{+}$,
and $\intco{a,b} = \emptyset$ if $a \geq b$.
$\intcc{a,b}_{\mathbb{Z}}$, $\intoo{a,b}_{\mathbb{Z}}$,
$\intco{a,b}_{\mathbb{Z}}$, and $\intoc{a,b}_{\mathbb{Z}}$ stand for
discrete intervals, e.g.
$\intcc{a,b}_{\mathbb{Z}} = \intcc{a,b} \cap \mathbb{Z}$.
We often drop the subscript ``$\mathbb{Z}$'' when
the type of interval to which we refer is obvious.

The set of $n \times m$-matrices over the field $\mathbb{F}$ is
denoted $\mathbb{F}^{n \times m}$, where
$\mathbb{F} \in \{ \mathbb{R}, \mathbb{C} \}$
throughout this note.
$X_{i,j}$ denotes the entry at position $(i,j)$ of
$X \in \mathbb{F}^{n \times m}$.
For any $x \in \mathbb{F}^n$ and $X \in \mathbb{F}^{n \times m}$,
$x^{\ast}$ and $X^{\ast}$ denote the transpose of $x$ and $X$,
respectively, if $\mathbb{F} = \mathbb{R}$, and the conjugate
transpose, if $\mathbb{F} = \mathbb{C}$.

\subsection{Systems, Solutions, Transition Matrices}
\label{ss:Systems}

The coefficient matrices in \ref{e:DiscreteTimeSystem},
\ref{e:ContinuousTimeSystem}, and \ref{e:OutputEquation},
$A(t)$, $B(t)$, $C(t)$ and $D(t)$,
are defined for all
$t \in \mathbb{Z}$ and all $t \in \mathbb{R}$, respectively. We
consider systems \begriff{over $\mathbb{R}$} and
\begriff{over $\mathbb{C}$}, so $A(t) \in \mathbb{F}^{n \times n}$,
$B(t) \in \mathbb{F}^{n \times r}$,
$C(t) \in \mathbb{F}^{m \times n}$, and
$D(t) \in \mathbb{F}^{m \times r}$
for each time $t$, where
$n \in \mathbb{N}$ and
$r, m \in \mathbb{Z}_{+}$.
The cases $r = 0$ and $m = 0$, which stand for systems without inputs
and outputs, respectively, are included here for the sake of
notational simplicity.

Given $u \colon \mathbb{Z} \to \mathbb{F}^r$, a map
$x \colon \intco{t_0,\infty}_{\mathbb{Z}} \to \mathbb{F}^n$ is a
\begriff{solution} of the system \ref{e:DiscreteTimeSystem}
(\begriff{generated} by the input signal $u$) if $t_0 \in \mathbb{Z}$
and \ref{e:DiscreteTimeSystem} holds for all
$t \in \intco{t_0,\infty}_{\mathbb{Z}}$.
Analogously, if $u \colon \mathbb{R} \to \mathbb{F}^r$, a map
$x \colon \intco{t_0,\infty} \to \mathbb{F}^n$ is a
\begriff{solution} of the system \ref{e:ContinuousTimeSystem}
(generated by $u$) if
$t_0 \in \mathbb{R}$, $x$ is absolutely continuous, and
\ref{e:ContinuousTimeSystem} holds for almost every (a.e.)
$t \in \intco{t_0,\infty}$, i.e., for all $t \in \intco{t_0,\infty}$
with the possible exception of a set of (Lebesgue) measure zero.
In the case of the system \ref{e:ContinuousTimeSystem} we will always
assume that the matrices $A$ and $B$
are locally integrable and
that input signals $u \colon \mathbb{R} \to \mathbb{F}^n$ are
measurable and locally essentially bounded. This hypothesis implies
both
existence and uniqueness of solutions \cite{Lukes82} and is
satisfied, e.g. if $A$, $B$
and $u$ are piecewise 
continuous.

The \begriff{general solution} of the system
\ref{e:DiscreteTimeSystem} and the system \ref{e:ContinuousTimeSystem}
is the map $\varphi$ defined by the requirement that for all
$x_0 \in \mathbb{F}^n$, $t_0$ and $u$,
$\varphi(\cdot,t_0,x_0,u)$ is the unique solution of
\ref{e:DiscreteTimeSystem} and \ref{e:ContinuousTimeSystem},
respectively, defined on $\intco{t_0,\infty}$ and satisfying
$\varphi(t_0,t_0,x_0,u) = x_0$.
Of course, we do not need to specify $u$ 
on the whole time axis, i.e., we define
$
\varphi(t,t_0,x_0,u|_{\intco{t_0,t}})
\defas
\varphi(t,t_0,x_0,u)
$,
where $t \geq t_0$ and $u|_{\intco{t_0,t}}$ denotes the
restriction of $u$ to the (discrete or continuous) interval
$\intco{t_0,t}$.
The map $\varphi(t,t_0,\cdot,0)$, which is linear, is called the
\begriff{transition matrix at $(t,t_0)$} of the system and is usually
denoted by $\Phi(t,t_0)$.
Then
\begin{align}
\label{e:VariationOfConstants:DiscreteTime}
\varphi(t,t_0,x_0,u)
&=
\Phi(t,t_0) x_0
+
\sum_{\tau=t_0}^{t-1}
\Phi(t,\tau+1) B(\tau) u(\tau),\\
\label{e:VariationOfConstants:ContinuousTime}
\varphi(t,t_0,x_0,u)
&=
\Phi(t,t_0) x_0
+
\int_{t_0}^t
\Phi(t,\tau) B(\tau) u(\tau)
d\tau
\end{align}
for the systems \ref{e:DiscreteTimeSystem} and
\ref{e:ContinuousTimeSystem}, respectively. We additionally have
$\Phi(t,t_0) = A(t-1) \cdot \ldots \cdot A(t_0)$ for the system
\ref{e:DiscreteTimeSystem}, and if $n = 1$ or $A$ is constant, then
$\Phi(t,t_0) = \exp\left( \int_{t_0}^t A(\tau) d \tau \right)$
for the system \ref{e:ContinuousTimeSystem}.
See \cite{Rugh96,Lukes82}.

\subsection{Nonzero Patterns and Graphs}
\label{ss:patterns}

We define the equivalence relation $\sim$ on $\mathbb{F}^{n\times m}$
by the requirement that $X \sim Y$ iff the positions of the zeros in
$X$ and $Y$ coincide, i.e., $X \sim Y$ iff $X_{i,j}=0$ implies
$Y_{i,j}=0$ and vice versa.
The equivalence classes
$[X]_{\mathalpha{\sim}} \in \mathalpha{\mathbb{F}^{n\times m}}/\mathalpha{\sim}$,
which we call \begriff{nonzero patterns}, or just \begriff{patterns},
will be represented by matrices
whose entries are asterisks and circles. Each
asterisk stands for a nonzero, and each circle, for a zero. For example, if the
coefficient matrices $A$ and $B$ in \ref{e:ContinuousTimeSystem} are
given by
\begin{equation}
\label{e:ex:POLY-SSC_but_not_ANALYTIC-SSC:Mat}
A(t)
=
\ifCLASSOPTIONdraftclsnofoot\arraycolsep1.6\arraycolsep\else\arraycolsep.6\arraycolsep\fi
\begin{pmatrix}
1 & 0 \\
0 & 0
\end{pmatrix}
\;\;\text{and}\;\;
B(t)
=
\begin{pmatrix}
\e^t \\
1
\end{pmatrix},
\end{equation}
then $[A(t)]_{\mathalpha{\sim}} = \mathcal{A}$ and $[B(t)]_{\mathalpha{\sim}} = \mathcal{B}$
for every $t$, where
\begin{equation}
\label{e:ex:POLY-SSC_but_not_ANALYTIC-SSC:Pat}
\mathcal{A}
=
\ifCLASSOPTIONdraftclsnofoot%
\arraycolsep.45\arraycolsep%
\def\arraystretch{.5}%
\else%
\arraycolsep.6\arraycolsep%
\fi
\begin{pmatrix}
\star & \circ \\
\circ & \circ
\end{pmatrix}
\;\;\text{and}\;\;
\mathcal{B}
=
\begin{pmatrix}
\star\\
\star
\end{pmatrix}.
\end{equation}
In particular,
$\mathcal{A}_{1,1} = \star$ and $\mathcal{A}_{1,2} = \circ$.
\begin{definition}
\label{def:OfPattern}
Let
$\mathcal{A} \in \mathbb{F}^{n \times n}/\mathalpha{\sim}$ and
$\mathcal{B} \in \mathbb{F}^{n \times r}/\mathalpha{\sim}$.
The system \ref{e:DiscreteTimeSystem}
\begriff{is of pattern $(\mathcal{A},\mathcal{B})$} if
\begin{equation}
\label{e:SystemOfPattern:Controllability}
A(t) \in \mathcal{A}
\text{\;\;and\;\;}
B(t) \in \mathcal{B}
\end{equation}
for all $t \in \mathbb{Z}$. Analogously,
the system
\ref{e:ContinuousTimeSystem} 
is of pattern $(\mathcal{A},\mathcal{B})$ if
$A$ and $B$
are locally integrable and
\ref{e:SystemOfPattern:Controllability}
holds for almost every $t \in \mathbb{R}$.
\end{definition}
\label{rev1:point1}
Some remarks are in order.
First of all, we would like to emphasize that 
the coefficient matrices
are required to satisfy the condition
\ref{e:SystemOfPattern:Controllability} (almost) everywhere on the
whole time axis for the sake of simplicity only. In any of our
subsequent results, these matrices actually need to be defined only on
$\intcc{t_0,t_1}$.
Still, there do exist systems that are not of any pattern, and their
controllability can not be decided from the results in the present
paper.
On the other hand, our notion of `pattern' is general enough to allow
nonzero matrix entries to change their signs. The sets of zeros of
such entries would be empty in the discrete-time case, but could be
non-empty of measure zero in the continuous-time case, where
the condition \ref{e:SystemOfPattern:Controllability}
is required to hold for almost everywhere $t$ rather than
for every $t$ and
Definition \ref{def:OfPattern} does not impose any continuity
requirements.
For example, a polynomial matrix entry $\mathbb{R} \to \mathbb{R}$ is
considered a nonzero entry iff it does not vanish identically on
$\mathbb{R}$.

The operations of addition, matrix composition and transposition for
patterns are defined by
\begin{align*}
[ X ]_{\mathalpha{\sim}} + [ Y ]_{\mathalpha{\sim}}
&=
[ |X| + |Y| ]_{\mathalpha{\sim}},\\
(
[ X ]_{\mathalpha{\sim}}, [ Y ]_{\mathalpha{\sim}}
)
&=
[
(X, Y)
]_{\mathalpha{\sim}},\\
[ X ]_{\mathalpha{\sim}}^{\ast}
&=
[ X^{\ast} ]_{\mathalpha{\sim}},
\end{align*}
whenever the operations on the right hand sides are defined.
Here,
$X$ and $Y$ are matrices,
$|X|$ denotes the matrix with entries $|X_{i,j}|$, and
$(X,Y)$ is the matrix consisting of the columns of
$X$ and $Y$.

Patterns of systems are conveniently represented by
graphs \cite{Reinschke88,Murota00,DionCommaultVanderWoude03}.
Specifically, if
$\mathcal{A} \in \mathbb{F}^{n \times n}/\mathalpha{\sim}$ and
$\mathcal{B} \in \mathbb{F}^{n \times r}/\mathalpha{\sim}$,
then the \begriff{graph} $\mathcal{G}(\mathcal{A},\mathcal{B})$ of
$(\mathcal{A},\mathcal{B})$ has vertices $1, \dots, n+r$, and there
is a (directed) edge from the vertex $v$ to the vertex $w$ if
$1 \leq w \leq n$ and $(\mathcal{A},\mathcal{B})_{w,v} = \star$.
In this case, $v$ is a \begriff{predecessor} of $w$, and $w$ is a
\begriff{successor} of $v$.
For any set $V$ of vertices,
$\Pre(V)$ denotes the set of predecessors of $V$, i.e.,
$v \in \Pre(V)$ if there exists an edge from $v$ to some vertex in
$V$.
Analogously, $\Post(V)$ denotes the set of successors of $V$.
The notations $\Pre(V)$ and $\Post(V)$ do not contain any reference to
the graph $\mathcal{G}(\mathcal{A},\mathcal{B})$, which will always be
clear from context.

\section{Controllability}
\label{s:MainResults}

In this section we present novel characterizations of controllability of
time-varying systems
in the strong structural sense.
We rely on controllability %
notions from \cite{Sontag98} throughout; for variants of
controllability concepts, see
e.g. \cite{KalmanHoNarendra63,CallierDesoer94,Rugh96}.

\begin{definition}
\label{def:Controllability}
Let $\Sigma$ denote
the discrete-time system \ref{e:DiscreteTimeSystem} or
the continuous-time system \ref{e:ContinuousTimeSystem}, assume
$t_0, t_1 \in \mathbb{Z}$ or $t_0, t_1 \in \mathbb{R}$, respectively,
and let $\varphi$ denote the general solution of $\Sigma$.

The pair $(t_0,x_0)$ \begriff{can be controlled to} the
pair $(t_1,x_1)$ if $x_0, x_1 \in \mathbb{F}^n$,
$t_0 \leq t_1$, and there exists a control input
$u \colon \intco{t_0,t_1} \to \mathbb{F}^r$ such that
$x_1 = \varphi(t_1, t_0, x_0, u)$.
The system $\Sigma$ is
\begriff{controllable on the interval \intcc{t_0,t_1}}
if $(t_0,x_0)$ can be controlled to $(t_1,x_1)$
for all $x_0, x_1 \in \mathbb{F}^n$, and
$\Sigma$ is
\begriff{controllable}
if for all $x_0, x_1 \in \mathbb{F}^n$ there
exist $\tau_0$ and $\tau_1$ such that
$(\tau_0,x_0)$ can be controlled to $(\tau_1,x_1)$.
\end{definition}

We will frequently need the well-known controllability criteria in
Proposition \ref{prop:Controllability} below, which follow
immediately from the formulas \ref{e:VariationOfConstants:DiscreteTime}
and \ref{e:VariationOfConstants:ContinuousTime}.
Here and in
the remainder of this note, $\Phi$ denotes the
transition matrix of the systems \ref{e:DiscreteTimeSystem} and
\ref{e:ContinuousTimeSystem}, in which it will always be clear from
context to which of the two systems we refer.

\begin{proposition}
\label{prop:Controllability}
Let $t_0, t_1 \in \mathbb{Z}$, $t_0 < t_1$.
Then the system \ref{e:DiscreteTimeSystem} is controllable on
$\intcc{t_0,t_1}$ iff the condition
\begin{equation}
\label{e:prop:Controllability:DiscreteTime}
p^{\ast} \Phi(t_1,\tau+1) B(\tau) = 0
\text{\ for every $\tau \in \intco{t_0,t_1}$}
\end{equation}
implies $p = 0$. Analogously, if 
$t_0, t_1 \in \mathbb{R}$, $t_0 < t_1$,
then the system \ref{e:ContinuousTimeSystem} is controllable on
$\intcc{t_0,t_1}$ iff the condition
\begin{equation}
\label{e:prop:Controllability:ContinuousTime}
p^{\ast} \Phi(t_1,\tau) B(\tau) = 0
\text{\ for a.e. $\tau \in \intcc{t_0,t_1}$}
\end{equation}
implies $p = 0$.
\end{proposition}

\subsection{Controllability of Time-Invariant Systems}
\label{ss:ControllabilityTimeInvariant}

Next, we review results for time-invariant systems.
If the systems \ref{e:DiscreteTimeSystem} and
\ref{e:ContinuousTimeSystem} are time-invariant,
the property of controllability of the
system \ref{e:DiscreteTimeSystem} (resp., the system
\ref{e:ContinuousTimeSystem}) on $\intcc{t_0,t_1}$ does not depend on
the actual times $t_0$ and $t_1$, provided that $t_0 + n \leq t_1$
(resp., $t_0 < t_1$). Moreover, the characterizations of
controllability in terms of the pair $(A(0),B(0))$ of matrices in the
discrete-time and the continuous-time case coincide. It is therefore
justified to call the pair $(A(0),B(0))$ \begriff{controllable} if the
time-invariant system \ref{e:DiscreteTimeSystem}, or, equivalently,
the time-invariant system \ref{e:ContinuousTimeSystem}, is so.
One of the well-known results for time-invariant systems, the
\begriff{Hautus criterion}, says that the pair
$(A,B) \in \mathbb{F}^{n \times n} \times \mathbb{F}^{n \times r}$
is controllable iff the (complex) matrix
\begin{equation}
\label{e:prop:LTIControllability:Hautus}
(\lambda \id - A, B ) \text{ is surjective}
\end{equation}
for every $\lambda \in \mathbb{C}$, where $\id$ denotes
the identity matrix.

The following result of \person{Mayeda} and \person{Yamada}
characterizes patterns that guarantee the controllability of
pairs of matrices. It has originally been established in
\cite{MayedaYamada79} under the additional assumption of input
accessibility, a minor restriction which has been removed in
\cite{i12str}.
In what follows,
we assume that
$\mathcal{A} \in \mathbb{F}^{n \times n} / \mathalpha{\sim}$
and
$\mathcal{B} \in \mathbb{F}^{n \times r} / \mathalpha{\sim}$,
unless specified otherwise.

\begin{theorem}
\label{th:MayedaYamada79}
Consider the conditions \ref{th:DiscreteTime:Main:GRAPH_zero} and
\ref{th:DiscreteTime:Main:GRAPH_nonzero} below.
\begin{enumerate}
\makeatletter
\def\theenumi{$G_\arabic{enumi}$}
\def\labelenumi{\textbf{\boldmath(\theenumi)}}
\def\p@enumi#1{(#1)}
\makeatother
\setcounter{enumi}{\theGraphConditionCtr}
\item
\label{th:DiscreteTime:Main:GRAPH_zero}
For every non-empty subset $V \subseteq \{1, \dots, n\}$
of vertices of $\mathcal{G}(\mathcal{A},\mathcal{B})$ there exists
a vertex $v \in \{1, \dots, n + r \}$ such that
$V \cap \Post(\{v\})$ is a singleton.
\item
\label{th:DiscreteTime:Main:GRAPH_nonzero}
For every non-empty subset $V \subseteq \{1, \dots, n\}$
of vertices of $\mathcal{G}(\mathcal{A},\mathcal{B})$ that satisfies
$V \subseteq \Pre(V)$ there exists
a vertex $v \in \{1, \dots, n + r \} \setminus V$ such that
$V \cap \Post(\{v\})$ is a singleton.
\setcounter{GraphConditionCtr}{\value{enumi}}
\end{enumerate}
\label{revitemWordOrder}
The condition \ref{th:DiscreteTime:Main:GRAPH_zero} is equivalent to
the requirement that \ref{e:prop:LTIControllability:Hautus} holds
for $\lambda = 0$ and every pair $(A,B)$ of pattern
$(\mathcal{A},\mathcal{B})$.
Analogously, the condition \ref{th:DiscreteTime:Main:GRAPH_nonzero}
is equivalent to the requirement that \ref{e:prop:LTIControllability:Hautus} holds
for every $\lambda \in \mathbb{C} \setminus \{ 0 \}$
and every pair $(A,B)$ of pattern $(\mathcal{A},\mathcal{B})$.
Thus, every pair $(A,B)$ of pattern $(\mathcal{A},\mathcal{B})$ is controllable iff both
\ref{th:DiscreteTime:Main:GRAPH_zero} and
\ref{th:DiscreteTime:Main:GRAPH_nonzero} hold.
\end{theorem}
\begin{figure}[t]
{%
\newcommand{\INPUT}{\item[\textbf{Input:}]}%
\newcommand{\OUTPUT}{\item[\textbf{Output:}]}%
\begin{algorithmic}[1]
\INPUT{$L$, $\mathcal{G}(\mathcal{A},\mathcal{B})$}
\REQUIRE {%
$L \in \{ 0, 1 \}$,
$\mathcal{A} \in \mathbb{F}^{n \times n}/\mathalpha{\sim}$,
$\mathcal{B} \in \mathbb{F}^{n \times r}/\mathalpha{\sim}$,
$n \in \mathbb{N}$,
$r \in \mathbb{Z}_{+}$%
}%
\STATE $V \defas \{ 1, \dots, n \}$
\WHILE {$V \not= \emptyset$}
\STATE {
\label{alg:StrongStructuralControllability:T1}
$T \defas \Menge{ v \in \Pre(V) }{ V \cap \Post(\{v\}) \text{ is a singleton}}$
\label{revcomplexstatement}
}
\IF {$L = 1$}
\STATE {
\label{alg:StrongStructuralControllability:T2}
$T \defas T \setminus V$
}
\ENDIF
\IF {$L = 0$ \OR $V \subseteq \Pre(V)$}
\IF {$T = \emptyset$}
\STATE {break} {\hspace{1em}\textit{// exit while loop}}
\label{revitemNoLine9alg}
\ENDIF
\STATE {
\label{alg:StrongStructuralControllability:Pick1}
Pick $v \in T$
}
\STATE {
\label{alg:StrongStructuralControllability:Remove1}
$V \defas V \setminus \Post(\{v\})$
}
\ELSE
\STATE {Pick $v \in V \setminus \Pre(V)$}
\STATE {
\label{alg:StrongStructuralControllability:Remove2}
$V \defas V \setminus \{ v \}$
}
\ENDIF
\ENDWHILE
\OUTPUT{$V$}
\end{algorithmic}%
}%
\caption{\label{alg:StrongStructuralControllability}
Algorithm for the verification of the condition
\ref{th:DiscreteTime:Main:GRAPH_zero} (if $L = 0$) and the condition
\ref{th:DiscreteTime:Main:GRAPH_nonzero} (if $L = 1$) in Theorem
\ref{th:MayedaYamada79}. It returns the empty set if the respective
condition holds, or a nonempty set $V$ for which the condition fails
to hold.}%
\end{figure}

The conditions \ref{th:DiscreteTime:Main:GRAPH_zero} and
\ref{th:DiscreteTime:Main:GRAPH_nonzero} can be verified using the
algorithm in \ref{alg:StrongStructuralControllability}, whose
correctness is immediate from the proof of the above
Theorem given in \cite[Section III]{i12str}.
We emphasize that the number of iterations performed by the algorithm
does not exceed the state space dimension $n$. Hence, despite the fact
that the conditions \ref{th:DiscreteTime:Main:GRAPH_zero} and
\ref{th:DiscreteTime:Main:GRAPH_nonzero} impose requirements on every
non-empty subset $V \subseteq \{1,\dots,n\}$, only at most $2n$ such
subsets are actually tested in the verification of
\ref{th:DiscreteTime:Main:GRAPH_zero} and
\ref{th:DiscreteTime:Main:GRAPH_nonzero}.
We remark in passing that the conditions \ref{th:DiscreteTime:Main:GRAPH_zero} and
\ref{th:DiscreteTime:Main:GRAPH_nonzero} are equivalent to
the possibility of transforming, through row
and column permutations, the patterns $(\mathcal{A},\mathcal{B})$ and
$( [ \id ]_{\mathalpha{\sim}} + \mathcal{A},\mathcal{B})$
into special forms.
See \cite{JarczykSvaricekAlt11,i12str}.

\begin{example}
\label{ex:Poljak92mod:Alg}
Consider the patterns $\mathcal{A}$ and $\mathcal{B}$ given by
\begin{equation}
\label{e:ex:Poljak92mod:pattern}
\mathcal{A}
=
\ifCLASSOPTIONdraftclsnofoot%
\arraycolsep.45\arraycolsep%
\def\arraystretch{.5}%
\else%
\arraycolsep.6\arraycolsep%
\fi
\begin{pmatrix}
\circ & \star  & \circ & \circ & \circ & \circ \\
 \star  & \circ & \circ & \circ & \circ & \circ \\
 \circ & \circ & \circ & \star  & \circ & \circ \\
 \circ & \circ & \circ & \circ & \star  & \circ \\
 \circ & \circ & \circ & \circ & \circ & \star  \\
 \circ & \circ & \circ & \circ & \circ & \circ
\end{pmatrix},
\;\;\;
\mathcal{B}
=
\begin{pmatrix}
\star  & \star  \\
 \star  & \circ \\
 \circ & \circ \\
 \star  & \circ \\
 \circ & \circ \\
 \circ & \star
\end{pmatrix}.
\end{equation}
See also \ref{fig:Poljak92mod_GAB}.
This example is a modification of the one in \cite{Poljak92}.
\begin{figure}[t]
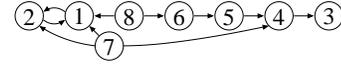

\centering 
\pgfdeclareimage[width=.99\linewidth]{Poljak92mod_GAB}{figures/Poljak92modGAB}%
\psfrag{1}[][]{1}
\psfrag{2}[][]{2}
\psfrag{3}[][]{3}
\psfrag{4}[][]{4}
\psfrag{5}[][]{5}
\psfrag{6}[][]{6}
\psfrag{7}[][]{7}
\psfrag{8}[][]{8}
\noindent
\parbox{.5\linewidth}{\pgfuseimage{Poljak92mod_GAB}}%
\caption{\label{fig:Poljak92mod_GAB}
Graph $\mathcal{G}{(\mathcal{A},\mathcal{B})}$ investigated in Example
\ref{ex:Poljak92mod:Alg}.%
}
\end{figure}
In order to verify that every pair $(A,B)$ of pattern
$(\mathcal{A},\mathcal{B})$ is
controllable, we first apply the algorithm in
\ref{alg:StrongStructuralControllability} with the parameter $L=0$.
Initially we have $V = \{1,\dots,6\}$, and $T$ is assigned the
value $\{1,2,4,5,6\}$ on line
\ref{alg:StrongStructuralControllability:T1}, which corresponds to the
columns of $(\mathcal{A},\mathcal{B})$ that contain exactly one
nonzero entry.
On line \ref{alg:StrongStructuralControllability:Pick1} we may choose
$v=1$, which results in the vertex $2$ being removed from $V$ on
line \ref{alg:StrongStructuralControllability:Remove1},
and vertices $1$, $3$, $4$, $5$ and $6$ may subsequently be
removed from $V$ on line
\ref{alg:StrongStructuralControllability:Remove1}, in this order.
Then $V = \emptyset$ on termination, and the condition
\ref{th:DiscreteTime:Main:GRAPH_zero} is satisfied.

Next, we apply the algorithm with the parameter $L=1$ to verify the
condition \ref{th:DiscreteTime:Main:GRAPH_nonzero}.
Then $T = \emptyset$ on line
\ref{alg:StrongStructuralControllability:T2}, and $V \setminus
\Pre(V)$ equals $\{ 3 \}$, which corresponds to the vanishing
third column of $\mathcal{A}$.
Subsequently, the vertices $3$, $4$, $5$, $6$ are removed
from $V$ on line \ref{alg:StrongStructuralControllability:Remove2},
then vertex $1$ is removed from $V$ on line
\ref{alg:StrongStructuralControllability:Remove1}. Finally, vertex
$2$ is removed from $V$ on line
\ref{alg:StrongStructuralControllability:Remove2}, so we arrive at
$V = \emptyset$ again. Hence, by Theorem
\ref{th:MayedaYamada79}, every pair $(A,B)$ of pattern
$(\mathcal{A},\mathcal{B})$ is controllable, regardless of the actual
numerical values at the nonzero locations in $A$ and $B$.
\end{example}

\subsection{Controllability of Discrete-Time Time-Varying Systems}
\label{ss:DiscreteTimeControllability}

We are now prepared to present and to prove our main result for
discrete-time systems.

\begin{theorem}
\label{th:DiscreteTime:Main}
Let
$t_0, t_1 \in \mathbb{Z}$, $t_0 < t_1$, and
consider the following condition.
\begin{enumerate}
\makeatletter
\def\theenumi{$G_\arabic{enumi}$}
\def\labelenumi{\textbf{\boldmath(\theenumi)}}
\def\p@enumi#1{(#1)}
\makeatother
\setcounter{enumi}{\theGraphConditionCtr}
\item
\label{th:DiscreteTime:Main:ShortInterval:GRAPH}
For every non-empty subset
$V \subseteq \{1, \dots, n(t_1 - t_0)\}$
of vertices of $\mathcal{G}(\mathcal{K})$ there exists
some vertex $v \in \{1, \dots, (n + r)(t_1 - t_0) \}$ such that $V$
contains exactly one successor of $v$ in
$\mathcal{G}(\mathcal{K})$, where
the pattern
$\mathcal{K} \in \mathbb{F}^{n(t_1-t_0) \times (n+r)(t_1-t_0)}/\mathalpha{\sim}$
is defined by
\begin{align*}
\mathcal{K}
&=
\ifCLASSOPTIONdraftclsnofoot%
\arraycolsep.45\arraycolsep%
\def\arraystretch{.5}%
\else%
\arraycolsep.6\arraycolsep%
\fi
\left.
\begin{pmatrix}
[0]_{\mathalpha{\sim}} & [\id]_{\mathalpha{\sim}} &        &              & \mathcal{B} &             &        &\\
         & \mathcal{A} & \ddots &              &             & \mathcal{B} &        &\\
         &             & \ddots & [\id]_{\mathalpha{\sim}}  &             &             & \ddots &\\
         &             &        &  \mathcal{A} &             &             & \phantom{\ddots}& \mathcal{B}
\end{pmatrix}
\right\}
n(t_1 - t_0)
\\[-3ex]
&\hphantom{=}
\ifCLASSOPTIONdraftclsnofoot%
\arraycolsep.45\arraycolsep%
\def\arraystretch{.5}%
\else%
\arraycolsep.6\arraycolsep%
\fi
\hspace*{1.1em}
\left.
\underbrace{%
\hphantom{%
\begin{matrix}
[0]_{\mathalpha{\sim}}
\end{matrix}}}_{n}
\underbrace{%
\hphantom{%
\begin{matrix}
[\id]_{\mathalpha{\sim}} & \ddots & [\id]_{\mathalpha{\sim}} &
\end{matrix}}}_{n(t_1 - t_0 - 1)}
\underbrace{%
\hphantom{%
\begin{matrix}
\mathcal{B}& \mathcal{B} & \ddots & \mathcal{B}&
\end{matrix}}}_{r(t_1 - t_0)}
\right.
\end{align*}
and the unspecified positions in $\mathcal{K}$ are occupied by
patterns $[0]_{\mathalpha{\sim}}$ of suitable sizes.
\setcounter{GraphConditionCtr}{\value{enumi}}
\end{enumerate}
Then the following holds.
\begin{enumerate}
\item
\label{th:DiscreteTime:Main:i}
Every system \ref{e:DiscreteTimeSystem}
of pattern $(\mathcal{A},\mathcal{B})$
is controllable on $\intcc{t_0,t_1}$ iff the condition
\ref{th:DiscreteTime:Main:ShortInterval:GRAPH} holds.
\item
\label{th:DiscreteTime:Main:ii}
If additionally $t_0 + n \leq t_1$, then every system
\ref{e:DiscreteTimeSystem} of pattern $(\mathcal{A},\mathcal{B})$
is controllable on $\intcc{t_0,t_1}$ iff
every time-invariant system \ref{e:DiscreteTimeSystem}
of pattern $(\mathcal{A},\mathcal{B})$ is so.
\popQED
\end{enumerate}
\end{theorem}

\begin{proof}
As for the latter claim, first observe that the condition is
obviously necessary. In order to prove that it is also sufficient,
assume that the system
\ref{e:DiscreteTimeSystem} is of nonzero pattern
$(\mathcal{A},\mathcal{B})$.
If $n = 1$, application of Theorem
\ref{th:MayedaYamada79} yields
$\mathcal{B} \not= \circ$,
and in particular,
$B(t_1 - 1) \not= 0$. Then the system \ref{e:DiscreteTimeSystem}
is controllable on $\intcc{t_0,t_1}$ by Proposition
\ref{prop:Controllability} since $\Phi(t_1,t_1) = 1$.

If $n > 1$, we assume that the theorem holds for all systems with
$(n-1)$-dimensional state space. We let $p \in \mathbb{F}^n$ satisfy
\ref{e:prop:Controllability:DiscreteTime} and show below
that then $p = 0$ necessarily, so that the system
\ref{e:DiscreteTimeSystem} is controllable on $\intcc{t_0,t_1}$ by
Proposition \ref{prop:Controllability}.

Let $V = \{ 1,\dots, n \}$ and observe that by
Theorem \ref{th:MayedaYamada79}, the conditions
\ref{th:DiscreteTime:Main:GRAPH_zero} and
\ref{th:DiscreteTime:Main:GRAPH_nonzero} hold. In particular, there
exists some vertex $v \in \{1, \dots, n + r \}$ such that
$V \cap \Post(\{v\})$ is a singleton.

Assume first that $v \notin V$. Then among the columns of
$\mathcal{B}$ there exists one with exactly one nonzero
component. So,
without loss of generality, $\mathcal{A}$ and
$\mathcal{B}$ can be partitioned according to
$\mathbb{F}^n = \mathbb{F}^{n-1} \times \mathbb{F}$,
\[
\mathcal{A}
=
\left(
\ifCLASSOPTIONdraftclsnofoot%
\arraycolsep.2\arraycolsep%
\def\arraystretch{.7}%
\else%
\arraycolsep.1\arraycolsep%
\fi
\begin{array}{ccc|c}
\rule[-1em]{0pt}{2.5em}& \mathcal{A}_{1,1} && \mathcal{A}_{1,2}\\
\hline
& \mathcal{A}_{2,1} && \mathcal{A}_{2,2}
\end{array}
\right),
\;\;
\mathcal{B}
=
\left(
\ifCLASSOPTIONdraftclsnofoot%
\arraycolsep.2\arraycolsep%
\def\arraystretch{.7}%
\else%
\arraycolsep.1\arraycolsep%
\fi
\begin{array}{ccc|c}
\rule[-1em]{0pt}{2.5em}& \mathcal{B}_{1,1} && \mathcal{B}_{1,2}\\
\hline
& \mathcal{B}_{2,1} && \mathcal{B}_{2,2}
\end{array}
\right),
\]
where
$\mathcal{B}_{2,2} = \star \in \mathbb{F} / \mathalpha{\sim}$,
$\mathcal{B}_{1,2} = [ 0 ]_{\mathalpha{\sim}} \in \mathbb{F}^{n-1} / \mathalpha{\sim}$,
and $\Phi$ as well as the coefficient matrices $A$ and $B$ are
partitioned analogously.
Moreover, since $\mathcal{B}_{1,2} = [ 0 ]_{\mathalpha{\sim}}$,
the conditions \ref{th:DiscreteTime:Main:GRAPH_zero} and
\ref{th:DiscreteTime:Main:GRAPH_nonzero} still hold when
$n$, $\mathcal{A}$ and $\mathcal{B}$ is replaced with $n-1$,
$\mathcal{A}_{1,1}$ and $(\mathcal{B}_{1,1},\mathcal{A}_{1,2})$,
respectively. Consequently, by our induction hypothesis, the following
system is controllable on $\intcc{t_0 + 1,t_1}$:
\begin{equation}
\label{e:th:DiscreteTime:Main:proof:dim=n-1}
x(t+1)
=
A_{1,1}(t) x(t) + B_{1,1}(t) u_1(t) + A_{1,2}(t) u_2(t).
\end{equation}

Let $p$ take the form
$p = (q,\alpha) \in \mathbb{F}^{n-1} \times \mathbb{F}$ and consider
the last column of $p^{\ast} \Phi(t_1,s+1) B(s)$. That column equals
$q^{\ast} B_{1,2}(s) + \alpha^{\ast} B_{2,2}(s)$ if $s = t_1 - 1$,
and hence, the condition \ref{e:prop:Controllability:DiscreteTime}
implies $\alpha = 0$ since
$B_{1,2} = 0$ and $B_{2,2}(t_1 - 1) \not= 0$. It follows that
\begin{equation}
\label{e:th:DiscreteTime:Main:proof:Phi12}
q^{\ast} \Phi_{1,2}(t_1,s + 1) = 0
\text{\ for all $s \in \intco{t_0,t_1}$}.
\end{equation}
Next define
$z(s) = q^{\ast} (\Phi_{1,1}(t_1,s) - \Psi(t_1,s))$, where $\Psi$ is
the transition matrix of the system
\ref{e:th:DiscreteTime:Main:proof:dim=n-1}.
Consider the adjoint equation
\begin{equation}
\label{e:DiscreteTimeAdjoint}
\Phi(t,s) = \Phi(t,s+1) A(s)
\end{equation}
of the system \ref{e:DiscreteTimeSystem}, which holds for all $t,s \in
\mathbb{Z}$ for which $s < t$, to see that
\ref{e:th:DiscreteTime:Main:proof:Phi12},
\ref{e:DiscreteTimeAdjoint} and the adjoint equation of the system
\ref{e:th:DiscreteTime:Main:proof:dim=n-1} imply that
$z(s) = z(s + 1) A_{1,1}(s)$ for all $s \in \intco{t_0,t_1}$.
From $z(t_1) = 0$ it follows that $z = 0$, hence
\begin{equation}
\label{e:th:DiscreteTime:Main:proof:Phi11}
q^{\ast} \Phi_{1,1}(t_1,s+1)
=
q^{\ast} \Psi(t_1,s+1)
\text{\ for all $s \in \intco{t_0,t_1}$}.
\end{equation}
Moreover, application of \ref{e:th:DiscreteTime:Main:proof:Phi12} and
\ref{e:th:DiscreteTime:Main:proof:Phi11} to the difference
equation for $\Phi_{1,2}(t_1,\cdot)$ that is part of the
adjoint equation \ref{e:DiscreteTimeAdjoint} yields
\begin{equation}
\label{e:th:DiscreteTime:Main:proof:PsiA12}
q^{\ast}
\Psi(t_1, s + 1) A_{1,2}(s)
=
0
\text{\ for all $s \in \intco{t_0 + 1,t_1}$},
\end{equation}
and condition
\ref{e:prop:Controllability:DiscreteTime} for
$p = (q,0)$, \ref{e:th:DiscreteTime:Main:proof:Phi12} and
\ref{e:th:DiscreteTime:Main:proof:Phi11} additionally show that
\begin{equation}
\label{e:th:DiscreteTime:Main:proof:PsiB11}
q^{\ast}
\Psi(t_1, s + 1) B_{1,1}(s)
=
0
\text{\ for all $s \in \intco{t_0,t_1}$}.
\end{equation}
As the system \ref{e:th:DiscreteTime:Main:proof:dim=n-1} is
controllable on $\intcc{t_0 + 1,t_1}$, it follows from
Prop.~\ref{prop:Controllability} and the identities
\ref{e:th:DiscreteTime:Main:proof:PsiA12} and
\ref{e:th:DiscreteTime:Main:proof:PsiB11} that $q=0$, hence $p=0$.

It remains to consider the case that $v$ cannot be chosen from the
complement of $V$. Then the condition
\ref{th:DiscreteTime:Main:GRAPH_nonzero}
implies $V \not\subseteq \Pre(V)$, so one of the columns of
$\mathcal{A}$ vanishes.
Hence, without loss of generality, $\mathcal{A}$ and
$\mathcal{B}$ can be partitioned according to
$\mathbb{F}^n = \mathbb{F}^{n-1} \times \mathbb{F}$,
\[
\mathcal{A}
=
\left(
\ifCLASSOPTIONdraftclsnofoot%
\arraycolsep.2\arraycolsep%
\def\arraystretch{.7}%
\else%
\arraycolsep.1\arraycolsep%
\fi
\begin{array}{ccc|c}
\rule[-1em]{0pt}{2.5em}& \mathcal{A}_{1,1} && \mathcal{A}_{1,2}\\
\hline
& \mathcal{A}_{2,1} && \mathcal{A}_{2,2}
\end{array}
\right),
\;\;
\mathcal{B}
=
\left(
\ifCLASSOPTIONdraftclsnofoot%
\arraycolsep.2\arraycolsep%
\def\arraystretch{.7}%
\else%
\arraycolsep.1\arraycolsep%
\fi
\begin{array}{cc}
\mathcal{B}_{1} & \rule[-1em]{0pt}{2.5em}\\
\hline
\mathcal{B}_{2}
\end{array}
\right),
\]
where
$\mathcal{A}_{2,2} = [ 0 ]_{\mathalpha{\sim}} \in \mathbb{F} / \mathalpha{\sim}$,
$\mathcal{A}_{1,2} = [ 0 ]_{\mathalpha{\sim}} \in \mathbb{F}^{n-1} / \mathalpha{\sim}$,
and $A$, $B$ and $\Phi$ are partitioned analogously.
Moreover, since $\mathcal{A}_{1,2} = [ 0 ]_{\mathalpha{\sim}}$,
the conditions \ref{th:DiscreteTime:Main:GRAPH_zero} and
\ref{th:DiscreteTime:Main:GRAPH_nonzero} still hold when
$n$, $\mathcal{A}$ and $\mathcal{B}$ is replaced with $n-1$,
$\mathcal{A}_{1,1}$ and $\mathcal{B}_{1}$,
respectively. Thus, %
by our induction hypothesis, the system
\begin{equation}
\label{e:th:DiscreteTime:Main:proof:dim=n-1:another}
x(t+1)
=
A_{1,1}(t) x(t) + B_{1}(t) u(t)
\end{equation}
is controllable on $\intcc{t_0,t_1 - 1}$.

Next, we observe that
$\mathcal{A}_{1,2} = [ 0 ]_{\mathalpha{\sim}}$
and
$\mathcal{A}_{2,2} = [ 0 ]_{\mathalpha{\sim}}$
imply $\Phi_{1,2} = 0$ and $\Phi_{1,1} = \Psi$, where $\Psi$ is
the transition matrix of the system
\ref{e:th:DiscreteTime:Main:proof:dim=n-1:another}. Therefore,
\ref{e:prop:Controllability:DiscreteTime} yields
$
p^{\ast} A_{\cdot,1}(t_1 - 1) \Psi(t_1 - 1, s + 1) B_1(s)
=0
$
for all $s \in \intco{t_0,t_1 - 1}$, where $A_{\cdot,1}(t_1 - 1)$
consists of the first $n-1$ columns of $A(t_1 - 1)$. Then
$p^{\ast} A(t_1 - 1) = 0$ by Proposition \ref{prop:Controllability}
since the system \ref{e:th:DiscreteTime:Main:proof:dim=n-1:another} is
controllable on $\intcc{t_0,t_1 - 1}$. Moreover, for $s = t_1 - 1$ the
identity \ref{e:prop:Controllability:DiscreteTime}
yields $p^{\ast} B(t_1 - 1) = 0$, so we arrive at
\begin{equation}
\label{e:th:DiscreteTime:Main:proof:61}
p^{\ast} (A(t_1 - 1),B(t_1 - 1)) = 0.
\end{equation}
By our assumption,
the pair
$(A(t_1 - 1),B(t_1 - 1))$ is controllable, so the coefficient
matrix in \ref{e:th:DiscreteTime:Main:proof:61} is surjective
by the Hautus criterion.
It follows that $p=0$ in the case that $v \in V$ either, which
completes our proof of the second claim of the theorem.

To prove the first claim, first observe that the condition
\ref{th:DiscreteTime:Main:ShortInterval:GRAPH} is equivalent to the
condition \ref{th:DiscreteTime:Main:GRAPH_zero} with
$(\mathcal{A},\mathcal{B})$, $n$ and $r$ replaced by
$([0]_{\mathalpha{\sim}},\mathcal{K})$,
$n (t_1 - t_0)$ and $r (t_1 - t_0)$, respectively. Hence, by Theorem
\ref{th:MayedaYamada79}, the condition
\ref{th:DiscreteTime:Main:ShortInterval:GRAPH}
holds iff every $K \in \mathcal{K}$ is surjective.
Next, we construct $K \in \mathcal{K}$ by
replacing the patterns $[0]_{\mathalpha{\sim}}$, $[\id]_{\mathalpha{\sim}}$, $\mathcal{A}$ and
$\mathcal{B}$ in the block row $i$ of $\mathcal{K}$,
$i \in \{1, \dots, t_1 - t_0\}$, by the matrices
$0$, $\id$, $A(t_0 + i -1)$ and $B(t_0 + i -1)$, respectively. 
Using elementary Gaussian operations in the same manner as in
the time-invariant case
\cite[Th.~6.2(iv), Ch.~2.6]{Rosenbrock70}, it follows that $K$ is
surjective iff the columns of $\Phi(t_1,s+1)B(s)$ for
$s \in \intco{t_0,t_1}$ span $\mathbb{F}^n$. In view of Proposition
\ref{prop:Controllability} and the observation that replacing the
blocks $\id$ in $K$ by nonsingular diagonal matrices corresponds to
suitably scaling the columns of $A(t)$, the proof is complete.
\end{proof}

For
discrete-time
time-varying systems \ref{e:DiscreteTimeSystem} on any
interval $\intcc{t_0,t_1}$, Theorem \ref{th:DiscreteTime:Main}
gives a complete characterization of nonzero patterns that guarantee
controllability on $\intcc{t_0,t_1}$. As we have observed in the proof
above, the condition \ref{th:DiscreteTime:Main:ShortInterval:GRAPH} is
equivalent to the condition \ref{th:DiscreteTime:Main:GRAPH_zero} with
$\mathcal{G}(\mathcal{A},\mathcal{B})$ replaced by
$\mathcal{G}([0]_{\mathalpha{\sim}},\mathcal{K})$, which
can be verified using the algorithm in
\ref{alg:StrongStructuralControllability}.
By the second claim of the Theorem, it
suffices to verify both \ref{th:DiscreteTime:Main:GRAPH_zero} and
\ref{th:DiscreteTime:Main:GRAPH_nonzero} instead if
$t_0 + n \leq t_1$, which is more efficient.
The latter assumption can not be dropped as shown by the following
example. In fact, the case $t_1 - t_0 < n$ remains open for
time-invariant systems
\ref{e:DiscreteTimeSystem}.

\begin{figure}[t]
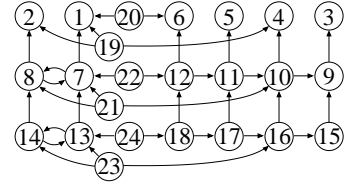

\centering 
\pgfdeclareimage[width=.99\linewidth]{Poljak92mod_K}{figures/Poljak92modK}%
\psfrag{1}[][]{1}
\psfrag{2}[][]{2}
\psfrag{3}[][]{3}
\psfrag{4}[][]{4}
\psfrag{5}[][]{5}
\psfrag{6}[][]{6}
\psfrag{7}[][]{7}
\psfrag{8}[][]{8}
\psfrag{9}[][]{9}
\psfrag{10}[][]{10}
\psfrag{11}[][]{11}
\psfrag{12}[][]{12}
\psfrag{13}[][]{13}
\psfrag{14}[][]{14}
\psfrag{15}[][]{15}
\psfrag{16}[][]{16}
\psfrag{17}[][]{17}
\psfrag{18}[][]{18}
\psfrag{19}[][]{19}
\psfrag{20}[][]{20}
\psfrag{21}[][]{21}
\psfrag{22}[][]{22}
\psfrag{23}[][]{23}
\psfrag{24}[][]{24}
\noindent
\parbox{.5\linewidth}{\pgfuseimage{Poljak92mod_K}}%
\caption{\label{fig:Poljak92mod_K}
Graph $\mathcal{G}(\mathcal{K})$ investigated in Example
\ref{ex:Poljak92mod}. The vertex set
$V = \{1,\dots,18\} \setminus \{3,4,9\}$ does not meet
the condition \ref{th:DiscreteTime:Main:ShortInterval:GRAPH}.%
}
\end{figure}

\begin{example}
\label{ex:Poljak92mod}
Let $\mathcal{A}$ and $\mathcal{B}$ be given by
\ref{e:ex:Poljak92mod:pattern} and
assume that the system \ref{e:DiscreteTimeSystem} is
of pattern $(\mathcal{A},\mathcal{B})$.
By Example \ref{ex:Poljak92mod:Alg},
the system is controllable if it
is time-invariant, and by Theorem \ref{th:DiscreteTime:Main},
the system is controllable in any case, on any interval
$\intcc{t_0,t_1}$ satisfying $t_0 + 6 \leq t_1$.
In the time-invariant case, let the $6 \times 6$-matrix
$M(t)$ consist of the columns
of $\Phi(t,s+1) B(s)$ for $s \in \intco{t-3,t}$.
Then the determinant $\det M(t)$ of $M(t)$ equals
$\pm
A_{1,2}^2 A_{2,1}^2 A_{3,4} A_{5,6} B_{1,2} B_{2,1} B_{4,1}^2 B_{6,2}^2$, so
the system is controllable on $\intcc{t_0,t_1}$ as soon as
$t_0 + 3 \leq t_1$. See
Proposition \ref{prop:Controllability}.
On the other hand, if the non-zeros in $A$ are all equal to
$1$ identically and the coefficient
$B$ is given by
$
B(t)^{\ast}
=
\left(
\begin{smallmatrix}
-1 & 3^{t/2} & 0 & 1 & 0 & 0 \\
 2 & 0 & 0 & 0 & 0 & 1
\end{smallmatrix}
\right)
$, then $\det M(t) = 0$ for all $t$, so the system is not
controllable on any interval
of the form
$\intcc{t_0,t_0+3}$.
This is consistent with the fact that the condition
\ref{th:DiscreteTime:Main:ShortInterval:GRAPH} is not satisfied if
$t_1 = t_0 + 3$.
See \ref{fig:Poljak92mod_K}.
\end{example}

\subsection{Controllability of Continuous-Time Time-Varying Systems}
\label{ss:ContinuousTimeControllability}

We have just demonstrated that in the discrete-time case,
controllability of all time-invariant systems of a given nonzero
pattern implies the controllability of all time-varying systems of
that pattern.
It turns out that the continuous-time case is quite different in that respect.

\begin{example}
\label{ex:LTI-SSC_but_not_POLY-SSC}
Consider the patterns $\mathcal{A}$ and $\mathcal{B}$ given by
\[
\mathcal{A}
=
\ifCLASSOPTIONdraftclsnofoot%
\arraycolsep.45\arraycolsep%
\def\arraystretch{.5}%
\else%
\arraycolsep.6\arraycolsep%
\fi
\begin{pmatrix}
\circ & \star & \circ \\
\circ & \circ & \star \\
\circ & \circ & \circ
\end{pmatrix}
\;\;\text{and}\;\;
\mathcal{B}
=
\begin{pmatrix}
\star\\
\circ\\
\star
\end{pmatrix}.
\]
The fact that every time-invariant system \ref{e:ContinuousTimeSystem}
of pattern $(\mathcal{A},\mathcal{B})$ is controllable follows from
Theorem \ref{th:MayedaYamada79}, or, alternatively, from the Hautus
criterion.
Now consider the time-varying system \ref{e:ContinuousTimeSystem} of
pattern $(\mathcal{A},\mathcal{B})$ in which the nonzero entries in
$A$ are all equal to $1$ identically and the coefficient $B$ is
given by
$
B(t)
=
(
t^2+1,
0,
-2
)$.
That system is not controllable since the choice
$p = (2, -2 t_1, t_1^2 + 1)$ satisfies the condition
\ref{e:prop:Controllability:ContinuousTime}
whenever $t_0 < t_1$.
\end{example}

\begin{example}
\label{ex:POLY-SSC_but_not_ANALYTIC-SSC}
Consider the patterns $\mathcal{A}$ and $\mathcal{B}$ given in
\ref{e:ex:POLY-SSC_but_not_ANALYTIC-SSC:Pat}.
As before, if the system \ref{e:ContinuousTimeSystem} is
time-invariant and of pattern $(\mathcal{A},\mathcal{B})$, it is
controllable. What is different here is that the same conclusion holds
if we merely assume the nonzero entries in
$A$ and $B$ to be polynomials rather than constants.
This fact is straightforward to verify.
Surprisingly, however, controllability is lost if we assume the
coefficients to be merely analytic. Indeed, if
$A$ and $B$ are given by
\ref{e:ex:POLY-SSC_but_not_ANALYTIC-SSC:Mat},
then the system \ref{e:ContinuousTimeSystem} is of pattern
$(\mathcal{A},\mathcal{B})$, yet it is not controllable as
the choice $p = (1, -\e^{t_1})$ satisfies condition
\ref{e:prop:Controllability:ContinuousTime}.
\end{example}

As the examples demonstrate, we not only need to distinguish
between time-invariant and time-varying systems, but we also have to
be precise about any regularity conditions imposed on the time-varying
coefficients of the system \ref{e:ContinuousTimeSystem}.
In this respect, one rather restrictive class of time-varying systems,
which is used
in the formulation of Theorem \ref{th:ContinuousTime:Main} below,
is that of \begriff{exponentially scaled} systems, by which we mean
systems \ref{e:ContinuousTimeSystem} over $\mathbb{F}$ that can be
transformed into a time-invariant system over $\mathbb{F}$ by means of
a time-varying change of coordinates of the form
$(t,x) \mapsto \exp(\Lambda t) x$, with
$\Lambda$ being diagonal. In other words, we require that there exist a
diagonal matrix $\Lambda \in \mathbb{F}^{n \times n}$ and matrices
$A_0 \in \mathbb{F}^{n \times n}$ and
$B_0 \in \mathbb{F}^{n \times r}$ such that
\begin{equation}
\label{e:EXP}
A(t) = \e^{-\Lambda t} A_0 \e^{\Lambda t} - \Lambda
\text{ and }
B(t) = \e^{-\Lambda t} B_0
\end{equation}
hold for every $t \in \mathbb{R}$.
It follows from \ref{e:EXP} that diagonal entries of $A$ may be added
or removed at will, without affecting any controllability properties
of the system \ref{e:ContinuousTimeSystem}. It is this fact that
distinguishes the discrete-time from the continuous-time case and
leads to more restrictive controllability conditions in the
time-varying  continuous-time case.

Our main result for the continuous-time case, presented below,
characterizes the nonzero patterns that ensure controllability of all
time-varying systems \ref{e:ContinuousTimeSystem}, or, what turns out
to be equivalent, of all exponentially scaled systems
\ref{e:ContinuousTimeSystem}. Obviously then, the result also
characterizes the patterns that ensure controllability of the systems
in any class in between the two extremes. Examples of such classes
include the ones defined by the requirement that certain entries of
the coefficient matrices $A$ and $B$ must be continuous, smooth, or of
constant sign.

\begin{theorem}
\label{th:ContinuousTime:Main}
Let
$t_0, t_1 \in \mathbb{R}$, $t_0 < t_1$, and consider
the following condition.
\begin{enumerate}
\makeatletter
\def\theenumi{$G_\arabic{enumi}$}
\def\labelenumi{\textbf{\boldmath(\theenumi)}}
\def\p@enumi#1{(#1)}
\makeatother
\setcounter{enumi}{\theGraphConditionCtr}
\item
\label{th:ContinuousTime:Main:GRAPH}
For every non-empty subset
$V \subseteq \{1, \dots, n\}$
of vertices of $\mathcal{G}(\mathcal{A},\mathcal{B})$ there exists
some vertex $v \in \{1, \dots, n + r \} \ \setminus V$ such that $V$
contains exactly one successor of $v$ in
$\mathcal{G}(\mathcal{A},\mathcal{B})$.
\setcounter{GraphConditionCtr}{\value{enumi}}
\end{enumerate}
The condition \ref{th:ContinuousTime:Main:GRAPH} is equivalent to
each of the following three statements.

\begin{enumerate}
\item
\label{th:ContinuousTime:Main:ALL}
Every system \ref{e:ContinuousTimeSystem}
of pattern $(\mathcal{A},\mathcal{B})$
is controllable on $\intcc{t_0,t_1}$.
\item
\label{th:ContinuousTime:Main:EXP}
Every exponentially scaled system \ref{e:ContinuousTimeSystem}
of pattern $(\mathcal{A},\mathcal{B})$
is controllable on $\intcc{t_0,t_1}$.
\item
\label{th:ContinuousTime:Main:LTI}
Every time-invariant system \ref{e:ContinuousTimeSystem}
of pattern
$( [ \id ]_{\mathalpha{\sim}} + \mathcal{A},\mathcal{B})$
is controllable.
\popQED
\end{enumerate}
\end{theorem}

The condition \ref{th:ContinuousTime:Main:GRAPH} is obviously
invariant with respect to the addition of any loops to the graph
$\mathcal{G}(\mathcal{A},\mathcal{B})$ and is thus equivalent to the
condition \ref{th:DiscreteTime:Main:GRAPH_nonzero} with
$[ \id ]_{\mathalpha{\sim}} + \mathcal{A}$
at the place of $\mathcal{A}$.
Note also that the patterns in conditions
\ref{th:ContinuousTime:Main:ALL} and \ref{th:ContinuousTime:Main:LTI}
differ.

\begin{proof}[Proof of Theorem \ref{th:ContinuousTime:Main}]
In the chain
\ref{th:ContinuousTime:Main:ALL}
\implies
\ref{th:ContinuousTime:Main:EXP}
\implies
\ref{th:ContinuousTime:Main:LTI}
\implies
\ref{th:ContinuousTime:Main:GRAPH}
\implies
\ref{th:ContinuousTime:Main:ALL}, the first implication is obvious,
and the third follows from Theorem \ref{th:MayedaYamada79} and the
remark preceding this proof.

\label{rev1:point2}
In order to prove that \ref{th:ContinuousTime:Main:LTI} follows from
\ref{th:ContinuousTime:Main:EXP}, we let
$A_0 \in [ \id ]_{\mathalpha{\sim}} + \mathcal{A}$ and
$B_0 \in \mathcal{B}$ and
define the diagonal matrix $\Lambda \in \mathbb{F}^{n \times n}$
by the requirement
\[
\Lambda_{i,i}
=
(A_0)_{i,i}
+
\begin{cases}
0, \text{if $\mathcal{A}_{i,i} = 0$},\\
1, \text{otherwise}
\end{cases}
\]
for all $i \in \{1, \dots, n\}$. If the coefficients $A$ and $B$
of the system \ref{e:ContinuousTimeSystem} satisfy \ref{e:EXP} for all
$t \in \mathbb{R}$, that system is exponentially
scaled and of nonzero pattern $(\mathcal{A},\mathcal{B})$, and hence,
it is controllable on $\intcc{t_0,t_1}$ by
\ref{th:ContinuousTime:Main:EXP}. Then, as
controllability is invariant with respect to the
change of coordinates given by \ref{e:EXP}, the pair $(A_0,B_0)$ is
controllable, which proves \ref{th:ContinuousTime:Main:LTI}.

In order to show that \ref{th:ContinuousTime:Main:GRAPH} implies
\ref{th:ContinuousTime:Main:ALL}, assume that the system
\ref{e:ContinuousTimeSystem} is of pattern
$(\mathcal{A},\mathcal{B})$.
If $n = 1$, application of
\ref{th:ContinuousTime:Main:GRAPH} yields
$\mathcal{B} \not= \circ$,
so $B(s) \not= 0$ for
a.e. $s \in \intcc{t_0,t_1}$. Then, since $\Phi(t_1,s) \not= 0$ for
all $s \in \intcc{t_0,t_1}$, the system \ref{e:ContinuousTimeSystem}
is controllable on $\intcc{t_0,t_1}$ by
Prop.~\ref{prop:Controllability}.

The proof in the case that $n > 1$ is analogous to that part of the
proof of
Theorem \ref{th:DiscreteTime:Main} where it is assumed that
$v \notin V$, with only three differences. Firstly, the
continuous-time system
\label{rev1:point4}
$
\dot x(t)
=
A_{1,1}(t) x(t) + B_{1,1}(t) u_1(t) + A_{1,2}(t) u_2(t)
$, which is controllable on $\intcc{t_0,t_1}$,
is used at the place of the system
\ref{e:th:DiscreteTime:Main:proof:dim=n-1}.
Secondly, from the fact that
$p = (q,\alpha) \in \mathbb{F}^{n-1} \times \mathbb{F}$
satisfies \ref{e:prop:Controllability:ContinuousTime}
we conclude that $\alpha = 0$ by the following argument.
The identity
\ref{e:prop:Controllability:ContinuousTime} yields
$\left(
q^{\ast} \Phi_{1,2}(t_1,s) + \alpha^{\ast} \Phi_{2,2}(t_1,s)
\right)
B_{2,2}(s)
= 0$
for a.e. $s \in \intcc{t_0,t_1}$, so
$\alpha = 0$ as $\Phi(t_1,t_1) = \id$, $\Phi$ is continuous, and
$B_{2,2} \not= 0$ a.e..
Finally, the
continuous-time
adjoint equation
$D_2 \Phi(t,s) = - \Phi(t,s) A(s)$, which holds for all
$t \in \mathbb{R}$ and a.e. $s \in \mathbb{R}$ \cite{Lukes82},
is used at the place of the discrete-time variant
\ref{e:DiscreteTimeAdjoint}. Here, $D_2$ denotes the partial
derivative with respect to the second argument.
Then $z=0$ as
$\dot z(s) = - z(s) A_{1,1}(s)$ holds for a.e. $s \in \intcc{t_0,t_1}$,
and we arrive at the identities
$q^{\ast} \Psi(t_1, s) A_{1,2}(s) = 0 $ and
$q^{\ast} \Psi(t_1, s) B_{1,1}(s) = 0$
for a.e. $s \in \intcc{t_0,t_1}$, at the place of
\ref{e:th:DiscreteTime:Main:proof:PsiA12} and
\ref{e:th:DiscreteTime:Main:proof:PsiB11}, to conclude that $p=0$.
\label{rev1:point2END}
\end{proof}

\section{Observability}
\label{ss:Observability}

In this section we present our results on strong structural observability
for both discrete-time and continuous-time systems. At the end of the
section, we point out a subtlety regarding the duality between
observability and controllability in the discrete-time case.

\begin{definition}
\label{def:Observability}
Let $\Sigma$ denote
the discrete-time system
\ref{e:DiscreteTimeSystem}, \ref{e:OutputEquation} or
the continuous-time system
\ref{e:ContinuousTimeSystem}, \ref{e:OutputEquation}, assume
$t_0, t_1 \in \mathbb{Z}$ or $t_0, t_1 \in \mathbb{R}$, respectively,
$t_0 \leq t_1$,
and let $\varphi$ denote the general solution of $\Sigma$.

The events $(t_0,x_0)$ and $(t_0,x_1)$ are
\begriff{indistinguishable on the interval $\intcc{t_0,t_1}$} if
$x_0, x_1 \in \mathbb{F}^n$
and for every control input
$u \colon \intco{t_0,t_1} \to \mathbb{F}^r$ and
every $t \in \intco{t_0,t_1}_{\mathbb{Z}}$
(resp., a.e. $t \in \intcc{t_0,t_1}$)
we have
$
C(t) \varphi(t,t_0,x_0,u) = C(t) \varphi(t,t_0,x_1,u)
$.

The system $\Sigma$ is
\begriff{observable on the interval $\intcc{t_0,t_1}$}
if $x_0 = x_1$ whenever $(t_0,x_0)$ and $(t_0,x_1)$ are
indistinguishable on $\intcc{t_0,t_1}$.
Finally, $\Sigma$ is \begriff{observable} if
indistinguishability of
$(\tau_0,x_0)$ and $(\tau_0,x_1)$ on the interval
$\intcc{\tau_0,\tau_1}$ for all $\tau_0$ and all $\tau_1 \geq \tau_0$
implies $x_0 = x_1$.
\end{definition}

\ifCLASSOPTIONdraftclsnofoot\looseness-1\fi
We note that our remarks on time-invariant systems at the beginning of
Section \ref{ss:ControllabilityTimeInvariant} equally apply to the
property of observability, and
we define patterns for systems with outputs in analogy to patterns for
systems \ref{e:DiscreteTimeSystem} and \ref{e:ContinuousTimeSystem}.
That is,
the system \ref{e:DiscreteTimeSystem}, \ref{e:OutputEquation}
\begriff{is of output pattern $(\mathcal{A},\mathcal{C})$} if
\begin{equation}
\label{e:SystemOfPattern:Observability}
A(t) \in \mathcal{A}
\text{\;\;and\;\;}
C(t) \in \mathcal{C}
\end{equation}
for all $t \in \mathbb{Z}$,
and the system
\ref{e:ContinuousTimeSystem}, \ref{e:OutputEquation}
is of output pattern $(\mathcal{A},\mathcal{C})$
if $A$ and $C$ are locally integrable and
\ref{e:SystemOfPattern:Observability} holds
for a.e. $t \in \mathbb{R}$.
Here and in what follows we assume that
$\mathcal{C} \in \mathbb{F}^{m \times n}/\mathalpha{\sim}$.

\begin{corollary}
\label{cor:DiscreteTime:Observability}
Let
$t_0, t_1 \in \mathbb{Z}$,
$t_0 < t_1$.
Then every system \ref{e:DiscreteTimeSystem}, \ref{e:OutputEquation}
of output pattern $(\mathcal{A},\mathcal{C})$
is observable on $\intcc{t_0,t_1}$ iff
the condition \ref{th:DiscreteTime:Main:ShortInterval:GRAPH}
holds with
$\mathcal{A}^{\ast}$, $\mathcal{C}^{\ast}$ and $m$ at the place of
$\mathcal{A}$, $\mathcal{B}$ and $r$, respectively.
Moreover, the latter
is equivalent to each of the following two statements under the
additional assumption $t_0 + n \leq t_1$.
\begin{enumerate}
\item
\label{cor:DiscreteTime:Observability:LTI}
Every time-invariant system
\ref{e:DiscreteTimeSystem}, \ref{e:OutputEquation} of output pattern
$(\mathcal{A},\mathcal{C})$ is observable.
\item
\label{cor:DiscreteTime:Observability:GRAPH}
The conditions \ref{th:DiscreteTime:Main:GRAPH_zero} and
\ref{th:DiscreteTime:Main:GRAPH_nonzero}
in Theorem \ref{th:MayedaYamada79} hold with
$\mathcal{A}^{\ast}$, $\mathcal{C}^{\ast}$ and $m$
at the place of
$\mathcal{A}$, $\mathcal{B}$ and $r$, respectively.
\popQED
\end{enumerate}
\end{corollary}

\begin{proof}
We assume without loss
of generality
that $B=0$ and $D=0$ in
\ref{e:DiscreteTimeSystem} and \ref{e:OutputEquation}.
Then the requirement that
$\widetilde A(t_1 - s) = A(t_0 + s)^{\ast}$
and
$\widetilde B(t_1 - s) = C(t_0 + s)^{\ast}$
for all $s \in \mathbb{Z}$
defines the adjoint system
\begin{equation}
\label{e:cor:DiscreteTime:Observability:proof:2}
x(t+1) = \widetilde A(t) x(t) + \widetilde B(t) u(t)
\end{equation}
for any given system \ref{e:DiscreteTimeSystem}, \ref{e:OutputEquation},
and vice versa, and the system
\ref{e:DiscreteTimeSystem}, \ref{e:OutputEquation} is of output
pattern $(\mathcal{A},\mathcal{C})$ (resp., time-invariant) iff the
system \ref{e:cor:DiscreteTime:Observability:proof:2} is of pattern
$(\mathcal{A}^{\ast},\mathcal{C}^{\ast})$ (resp., time-invariant).
In addition, the condition \ref{e:prop:Controllability:DiscreteTime} for
the system \ref{e:cor:DiscreteTime:Observability:proof:2} is
equivalent to the condition that
$C(s) \Phi(s,t_0) p = 0$ for every $s \in \intco{t_0,t_1}$, where
$\Phi$ is the transition matrix of the system
\ref{e:DiscreteTimeSystem}.
The latter condition implies $p = 0$ iff
the system \ref{e:DiscreteTimeSystem}, \ref{e:OutputEquation} is
observable on $\intcc{t_0,t_1}$ \cite[Th.~25.9]{Rugh96}.
Hence, on the interval $\intcc{t_0,t_1}$,
the system \ref{e:DiscreteTimeSystem}, \ref{e:OutputEquation}
is observable iff
the adjoint system \ref{e:cor:DiscreteTime:Observability:proof:2}
is controllable, and an application of
Theorem \ref{th:DiscreteTime:Main} completes the proof.
\end{proof}

In order to be able to present our continuous-time
observability result we extend the notion of \begriff{exponentially
  scaled system} to systems with outputs in the obvious way: The
system \ref{e:ContinuousTimeSystem}, \ref{e:OutputEquation} is
exponentially scaled if there exist a
diagonal matrix $\Lambda \in \mathbb{F}^{n \times n}$ and matrices
$A_0 \in \mathbb{F}^{n \times n}$,
$B_0 \in \mathbb{F}^{n \times r}$
and $C_0 \in \mathbb{F}^{m \times n}$ such that both \ref{e:EXP} and
$C(t) = C_0 \exp(\Lambda t)$ hold for every $t \in\mathbb{R}$.

\begin{corollary}
\label{cor:ContinuousTime:Observability}
Let
$t_0, t_1 \in \mathbb{R}$, $t_0 < t_1$.
Then the following four statements are equivalent:
\begin{enumerate}
\item
\label{cor:ContinuousTime:Observability:ALL}
Every system \ref{e:ContinuousTimeSystem}, \ref{e:OutputEquation}
of output pattern $(\mathcal{A},\mathcal{C})$
is observable on $\intcc{t_0,t_1}$.
\item
\label{cor:ContinuousTime:Observability:EXP}
Every exponentially scaled system
\ref{e:ContinuousTimeSystem}, \ref{e:OutputEquation}
of output pattern $(\mathcal{A},\mathcal{C})$
is observable on $\intcc{t_0,t_1}$.
\item
\label{cor:ContinuousTime:Observability:LTI}
Every time-invariant system
\ref{e:ContinuousTimeSystem}, \ref{e:OutputEquation}
of output pattern
$( [ \id ]_{\mathalpha{\sim}} + \mathcal{A},\mathcal{C})$
is observable.
\item
\label{cor:ContinuousTime:Observability:GRAPH}
Condition \ref{th:ContinuousTime:Main:GRAPH}
holds with
$(\mathcal{A}^{\ast},\mathcal{C}^{\ast})$ and $m$
at the place of
$(\mathcal{A},\mathcal{B})$ and $r$, respectively.
\popQED
\end{enumerate}
\end{corollary}

The proof of Corollary \ref{cor:ContinuousTime:Observability} is
omitted as it is analogous to that of Corollary
\ref{cor:DiscreteTime:Observability}. In particular, the adjoint
system $\dot x(t) = - A(t)^{\ast} x(t) + C(t)^{\ast} u(t)$, which is
controllable on $\intcc{t_0,t_1}$ iff
the system \ref{e:ContinuousTimeSystem}, \ref{e:OutputEquation} is
observable on $\intcc{t_0,t_1}$, see \cite{Sontag98}, is used at the
place of the system \ref{e:cor:DiscreteTime:Observability:proof:2}.

Given the duality between our controllability and our observability
results, we warn against a subtlety with
the discrete-time case. While it follows from Theorem
\ref{th:DiscreteTime:Main} and Corollary
\ref{cor:DiscreteTime:Observability} that \emph{every} system
\ref{e:DiscreteTimeSystem}, \ref{e:OutputEquation}
of output pattern $(\mathcal{A},\mathcal{C})$
is observable on $\intcc{t_0,t_1}$
iff \emph{every} system
\begin{equation}
\label{e:FalseDiscreteTimeAdjoint}
x(t + 1) = A(t)^{\ast} x(t) + C(t)^{\ast} u(t)
\end{equation}
of pattern $(\mathcal{A}^{\ast},\mathcal{C}^{\ast})$
is controllable on $\intcc{t_0,t_1}$, the respective properties of the
two systems are not, in general, equivalent.

\begin{example}
\label{ex:DiscreteTimeDualityGap}
The system \ref{e:DiscreteTimeSystem}, \ref{e:OutputEquation} is
not observable if
\[
A(t)
=
\ifCLASSOPTIONdraftclsnofoot\arraycolsep1.6\arraycolsep\else\arraycolsep.6\arraycolsep\fi
\begin{pmatrix}
-2 & \e^{1-t} \\
 -3 \e^t & 2 \e
\end{pmatrix}
\;\;\text{and}\;\;
C(t)
=
(\e^t,-\e)
\]
for all $t \in \mathbb{Z}$,
whereas the system \ref{e:FalseDiscreteTimeAdjoint} is
controllable on $\intcc{t_0,t_1}$ whenever $t_0 + 2 \leq t_1$.
Note also that $A(t)$ is invertible for all $t$.
\end{example}
\vspace*{-1.15\baselineskip}
\bibliographystyle{IEEEtran}
\bibliography{IEEEtranBSTCTL,preambles,mrabbrev,strings,fremde,eigeneCONF,eigeneJOURNALS,eigenePATENT,eigeneREPORTS,eigeneTALKS,eigeneTHESES}

%
\def\ocirc#1{\ifmmode\setbox0=\hbox{$#1$}\dimen0=\ht0 \advance\dimen0
  by1pt\rlap{\hbox to\wd0{\hss\raise\dimen0
  \hbox{\hskip.2em$\scriptscriptstyle\circ$}\hss}}#1\else {\accent"17 #1}\fi}
  \def\cprime{$'$} \ifx\hyperbaseurl\undefined
  \def\href#1#2{#2}\def\url#1{\texttt{#1}} \fi
  \ifx\ExplicitURLsInBibTeX\undefined\relax\else
  \def\href#1#2{www.reiszig.de/gunther/#1}\def\url#1{\texttt{#1}} \fi
\begin{thebibliography}{10}
\providecommand{\url}[1]{#1}
\csname url@samestyle\endcsname
\providecommand{\newblock}{\relax}
\providecommand{\bibinfo}[2]{#2}
\providecommand{\BIBentrySTDinterwordspacing}{\spaceskip=0pt\relax}
\providecommand{\BIBentryALTinterwordstretchfactor}{4}
\providecommand{\BIBentryALTinterwordspacing}{\spaceskip=\fontdimen2\font plus
\BIBentryALTinterwordstretchfactor\fontdimen3\font minus
  \fontdimen4\font\relax}
\providecommand{\BIBforeignlanguage}[2]{{%
\expandafter\ifx\csname l@#1\endcsname\relax
\typeout{** WARNING: IEEEtran.bst: No hyphenation pattern has been}%
\typeout{** loaded for the language `#1'. Using the pattern for}%
\typeout{** the default language instead.}%
\else
\language=\csname l@#1\endcsname
\fi
#2}}
\providecommand{\BIBdecl}{\relax}
\BIBdecl

\bibitem{MayedaYamada79}
H.~Mayeda and T.~Yamada, ``Strong structural controllability,'' \emph{SIAM J.
  Control Optim.}, vol.~17, no.~1, pp. 123--138, 1979.

\bibitem{Reinschke88}
K.~Reinschke, \emph{Multivariable Control -- A Graph Theoretic Approach},
  Lect. Notes Control Inform. Sciences.\hskip 1em plus 0.5em minus 0.4em\relax
  Springer, 1988, vol. 108.

\bibitem{Murota00}
K.~Murota, \emph{Matrices and matroids for systems analysis}.\hskip 1em plus
  0.5em minus 0.4em\relax Springer, 2000.

\bibitem{DionCommaultVanderWoude03}
J.-M. Dion, C.~Commault, and J.~van~der Woude, ``Generic properties and control
  of linear structured systems: A survey.'' \emph{Automatica J. IFAC}, vol.~39,
  no.~7, pp. 1125--1144, 2003.

\bibitem{i03diagnosis}
\BIBentryALTinterwordspacing
G.~Rei{\ss}ig and U.~Feldmann, ``A simple and general method for detecting
  structural inconsistencies in large electrical networks,'' \emph{IEEE Trans.
  Circuits Systems I Fund. Theory Appl.}, vol.~50, no.~11, pp. 1482--1485, Nov.
  2003.
  Available: \url{http://dx.doi.org/10.1109/TCSI.2003.818620}
\BIBentrySTDinterwordspacing

\bibitem{LiuLinChen13b}
X.~Liu, H.~Lin, and B.~M. Chen, ``Graph-theoretic characterisations of
  structural controllability for multi-agent system with switching topology,''
  \emph{Internat. J. Control}, vol.~86, no.~2, pp. 222--231, 2013.

\bibitem{LouHong12}
Y.~Lou and Y.~Hong, ``Controllability analysis of multi-agent systems with
  directed and weighted interconnection,'' \emph{Internat. J. Control},
  vol.~85, no.~10, pp. 1486--1496, 2012.

\bibitem{JarczykSvaricekAlt11}
J.~C. Jarczyk, F.~Svaricek, and B.~Alt, ``Strong structural controllability of
  linear systems revisited,'' in \emph{Proc. 50th IEEE Conf. Decision and
  Control (CDC) and European Control Conference (\nobreak{ECC}), Orlando, FL,
  U.S.A., 12-15 Dec. 2011}, pp. 1213--1218.

\bibitem{HashimotoAmemiya11}
T.~Hashimoto and T.~Amemiya, ``Controllability and observability of linear
  time-invariant uncertain systems irrespective of bounds of uncertain
  parameters,'' \emph{IEEE Trans. Automat. Control}, vol.~56, no.~8, pp.
  1807--1817, 2011.

\bibitem{ChapmanMesbahi13}
A.~Chapman and M.~Mesbahi, ``Strong structural controllability of networked
  dynamics,'' in \emph{Proc. American Control Conference (ACC), Washington, DC,
  U.S.A., 17-19 Jun. 2013}, pp. 6141--6146.

\bibitem{i12str}
\BIBentryALTinterwordspacing
C.~Hartung, G.~Rei{\ss}ig, and F.~Svaricek, ``Characterization of strong
  structural controllability of uncertain linear time-varying discrete-time
  systems,'' in \emph{Proc. 51st IEEE Conf. Decision and Control (CDC), Maui,
  Hawaii, U.S.A., 10-13 Dec. 2012}, pp. 2189--2194. Available:
  \url{http://dx.doi.org/10.1109/CDC.2012.6426326}
\BIBentrySTDinterwordspacing

\bibitem{i13str}
C.~Hartung, G.~Rei{\ss}ig, and F.~Svaricek, ``Sufficient conditions for strong
  structural controllability of uncertain linear time-varying systems,'' in
  \emph{Proc. American Control Conference (ACC), Washington, DC, U.S.A., 17-19
  Jun. 2013}, pp. 5895--5900.

\bibitem{i13strb}
C.~Hartung, G.~Rei{\ss}ig, and F.~Svaricek, ``Necessary conditions for
  structural and strong structural controllability of linear time-varying
  systems,'' in \emph{Proc. European Control Conference (\nobreak{ECC}),
  Z{\"u}rich, Switzerland, 17-19 Jul. 2013}, pp. 1335--1340.

\bibitem{Lukes82}
D.~L. Lukes, \emph{Differential equations}, Academic Press, 1982.

\bibitem{Rugh96}
W.~J. Rugh, \emph{Linear system theory}, 2nd~ed.\hskip 1em plus 0.5em minus
  0.4em\relax Prentice Hall, 1996.

\bibitem{Sontag98}
E.~D. Sontag, \emph{Mathematical control theory}, 2nd~ed., Springer, 1998.

\bibitem{KalmanHoNarendra63}
R.~E. Kalman, Y.~C. Ho, and K.~S. Narendra, ``Controllability of linear
  dynamical systems,'' \emph{Contributions to Differential Equations}, vol.~1,
  pp. 189--213, 1963.

\bibitem{CallierDesoer94}
\BIBentryALTinterwordspacing
F.~M. Callier and C.~A. Desoer, \emph{Linear system theory}, Springer, 1991.
\BIBentrySTDinterwordspacing

\bibitem{GolumbicHirstLewenstein01}
M.~C. Golumbic, T.~Hirst, and M.~Lewenstein, ``Uniquely restricted matchings,''
  \emph{Algorithmica}, vol.~31, no.~2, pp. 139--154, 2001.

\bibitem{Poljak92}
S.~Poljak, ``On the gap between the structural controllability of time-varying
  and time-invariant systems,'' \emph{IEEE Trans. Automat. Control}, vol.~37,
  no.~12, pp. 1961--1965, 1992.

\bibitem{Rosenbrock70}
H.~H. Rosenbrock, \emph{State-space and multivariable theory}.\hskip 1em plus
  0.5em minus 0.4em\relax Wiley 1970.

\end{thebibliography}
\end{document}